\documentclass[reqno]{amsart}
\usepackage{amsmath, amsthm, amssymb, amstext}

\usepackage{hyperref,xcolor}
\hypersetup{
 pdfborder={0 0 0},
 colorlinks,
}
\usepackage{enumitem}
\allowdisplaybreaks

\newtheorem{theorem}{Theorem}
\newtheorem{remark}[theorem]{Remark}
\newtheorem{lemma}[theorem]{Lemma}
\newtheorem{proposition}[theorem]{Proposition}
\newtheorem{corollary}[theorem]{Corollary}

\DeclareMathOperator*{\divergenz}{div}              %
\DeclareMathOperator*{\Ss}{S}

\newcommand{\N}{\mathbb{N}}

\newcommand{\R}{\mathbb{R}}

\newcommand{\Lp}[1]{L^{#1}(\Omega)}

\newcommand{\Wp}[1]{W^{1,#1}(\Omega)}

\newcommand{\Wpzero}[1]{W^{1,#1}_0(\Omega)}

\newcommand{\lan}{\langle}
\newcommand{\ran}{\rangle}
\newcommand{\eps}{\varepsilon}
\newcommand{\ph}{\varphi}

\newcommand{\into}{\int_{\Omega}}

\newcommand{\weak}{\rightharpoonup}

\newcommand{\Linf}{L^{\infty}(\Omega)}
\newcommand{\close}{\overline{\Omega}}

\renewcommand{\l}{\left}
\renewcommand{\r}{\right}

\newcommand{\WH}{W^{1, \mathcal{H}}_0(\Omega)}
\numberwithin{theorem}{section}
\numberwithin{equation}{section}

\newcommand{\nc}{\mathcal{N}}
\newcommand{\e}{\mathbb \varepsilon}

\def\le{\leqslant}

\def\phi{\varphi}

\def\ykh#1{\left(#1\right)}

\def\eh{\mathcal{H}}

\title[Existence of solutions for singular double phase problems]{Existence of solutions for singular double phase problems via the Nehari manifold method}

\author[W.\,Liu]{Wulong Liu}
\address[W.\,Liu]{School of Science, Jiangxi University of Science and Technology, Ganzhou, Jiangxi 341000, PR China}
\email{liuwul000@gmail.com}

\author[G.\,Dai]{Guowei Dai}
\address[G.\,Dai]{School of Mathematical Sciences, Dalian University of Technology, Dalian, 116024, PR China}
\email{daiguowei@dlut.edu.cn}

\author[N.\,S.\,Papageorgiou]{Nikolaos S.\,Papageorgiou}
\address[N.\,S.\,Papageorgiou]{National Technical University, Department of Mathematics, Zografou Campus, Athens 15780, Greece}
\email{npapg@math.ntua.gr}

\author[P.\,Winkert]{Patrick Winkert}
\address[P.\,Winkert]{Technische Universit\"{a}t Berlin, Institut f\"{u}r Mathematik, Stra\ss e des 17.\,Juni 136, 10623 Berlin, Germany}
\email{winkert@math.tu-berlin.de}


\begin{document}

\begin{abstract}
	In this paper we study quasilinear elliptic equations driven by the double phase operator and a right-hand side which has the combined effect of a singular and of a parametric term. Based on the fibering method by using the Nehari manifold we are going to prove the existence of at least two weak solutions for such problems when the parameter is sufficiently small.
\end{abstract}

\maketitle

\section{Introduction}\label{section_1}

Zhikov \cite{Zhikov-1986} was the first who introduced and studied functionals whose integrands change their ellipticity according to a point in order to provide models for strongly anisotropic materials. As a prototype he considered the functional
\begin{align}\label{integral_minimizer}
	\omega \mapsto \int_\Omega \big(|\nabla  \omega|^p+\mu(x)|\nabla  \omega|^q\big)\,\mathrm{d} x,
\end{align}
where $1<p<q$ and with a nonnegative weight function $\mu\in\Linf$. Therefore, the integrand of \eqref{integral_minimizer} has unbalanced growth. The main feature of the functional defined in \eqref{integral_minimizer} is the change of ellipticity on the set where the weight function is zero, that is, on the set $\{x\in \Omega: \mu(x)=0\}$. In other words, the energy density of \eqref{integral_minimizer} exhibits ellipticity in the gradient of order $q$ on the points $x$ where $\mu(x)$ is positive and of order $p$ on the points $x$ where $\mu(x)$ vanishes. We also refer to the book of Zhikov-Kozlov-Ole\u{\i}nik \cite{Zhikov-Kozlov-Oleinik-1994}. Functionals of type \eqref{integral_minimizer} have been intensively studied in the past decade concerning regularity for isotropic and anisotropic settings. We mention the papers of Baroni-Colombo-Mingione \cite{Baroni-Colombo-Mingione-2015,Baroni-Colombo-Mingione-2016,Baroni-Colombo-Mingione-2018}, Baroni-Kuusi-Mingione \cite{Baroni-Kuusi-Mingione-2015}, Byun-Oh \cite{Byun-Oh-2017,Byun-Oh-2020}, Byun-Youn \cite{Byun-Youn-2018}, Colombo-Mingione \cite{Colombo-Mingione-2015a,Colombo-Mingione-2016, Colombo-Mingione-2015b}, De Filippis-Mingione \cite{De-Filippis-Mingione-2020,De-Filippis-Mingione-2021, De-Filippis-Mingione-2020b}, De Filippis-Palatucci \cite{De-Filippis-Palatucci-2019}, Esposito-Leonetti-Mingione \cite{Esposito-Leonetti-Mingione-2004}, Esposito-Leonetti-Petricca \cite{Esposito-Leonetti-Petricca-2019}, 
Marcellini \cite{Marcellini-1991,Marcellini-1989b, Marcellini-1989}, Ok \cite{Ok-2018,Ok-2020}, Ragusa-Tachikawa \cite{Ragusa-Tachikawa-2020}, Riey \cite{Riey-2019} and the references therein.

The energy functional \eqref{integral_minimizer} is related to the so-called double phase operator which is defined by
\begin{align}\label{operator_double_phase}
	\divergenz \big(|\nabla u|^{p-2} \nabla u+ \mu(x) |\nabla u|^{q-2} \nabla u\big)\quad \text{for }u\in \WH
\end{align}
with an appropriate Musielak-Orlicz Sobolev space $\WH$, see its definition in Section \ref{section_2}. It is easy to see that \eqref{operator_double_phase} reduces to the $p$-Laplacian if $\mu\equiv 0$ or to the weighted $(q,p)$-Laplacian if $\inf_{\close} \mu\geq \mu_0>0$, respectively.

Given a bounded domain $\Omega\subset \R^N $, $N\geq 2$, with Lipschitz boundary $\partial \Omega$, we study the following singular double phase problem
\begin{equation}\label{problem}
	\begin{aligned}
		-\divergenz\big(|\nabla u|^{p-2} \nabla u+ \mu(x) |\nabla u|^{q-2} \nabla u\big)&= a(x)u^{-\gamma}+\lambda u^{r-1}\quad && \text{in } \Omega, \\
		u&= 0 && \text{on } \partial\Omega,
	\end{aligned}
\end{equation}
where we suppose the subsequent assumptions:
\begin{enumerate}
	\item[\textnormal{(H):}]
	\begin{enumerate}[itemsep=0.2cm]
		\item[\textnormal{(i)}]
			$1<p<N$, $p<q<p^*=\frac{Np}{N-p}$ and $0 \leq \mu(\cdot)\in L^\infty(\Omega)$;
		\item[\textnormal{(ii)}]
		$0<\gamma<1$ and $q<r<p^*$;
		\item[\textnormal{(iii)}]
		$a \in \Linf$ and $a(x) > 0$ for a.\,a.\,$x\in\Omega$.
	\end{enumerate}
\end{enumerate}

A function $u \in\WH$ is said to be a weak solution if $a(\cdot)u^{-\gamma} h \in \Lp{1}$, $u> 0$ for a.\,a.\,$x\in\Omega$ and 
\begin{equation}\label{weak_solution}
	\begin{split}
		& \into \big(|\nabla u|^{p-2} \nabla u+ \mu(x) |\nabla u|^{q-2} \nabla u\big) \cdot \nabla h \,\mathrm{d} x\\
		&= \into a(x)u^{-\gamma} h \,\mathrm{d} x+\lambda \into u^{r-1}h\,\mathrm{d} x
	\end{split}
\end{equation}
is satisfied for all $h \in \WH$. Due to \textnormal{(H)(ii)} we see that $\into u^{r-1}h\,\,\mathrm{d} x$ is finite since $r<p^*$. So, the definition of a weak solution in \eqref{weak_solution} is well-defined. The corresponding energy functional  $\ph_\lambda\colon\WH\to \R$  for problem \eqref{problem} is given by
\begin{align*}
		\ph_{\lambda}(u)=\frac{1}{p} \|\nabla u\|_p^p+\frac{1}{q}\|\nabla u\|_{q,\mu}^q-\frac{1}{1-\gamma}\into a(x)|u|^{1-\gamma}\,\mathrm{d} x-\frac{\lambda}{r}\|u\|_r^r.
\end{align*}

The main result in this paper is the following theorem.

\begin{theorem}\label{main_result}
	Let hypotheses \textnormal{(H)} be satisfied. Then there exists $\hat{\lambda}_0^*>0$ such that for all $\lambda \in (0,\hat{\lambda}_0^*]$ problem \eqref{problem} has at least two weak solutions $u^*, v^* \in \WH$ such that $\ph_\lambda(u^*)<0<\ph_\lambda(v^*)$.
\end{theorem}

The main characteristic in our treatment is the usage of the so-called Nehari manifold which turned into a very powerful tool in order to find solutions for differential equations via critical point theory. This method was first introduced by Nehari \cite{Nehari-1961, Nehari-1960} and the idea behind is the following: For a real reflexive Banach space $X$ and a functional $\Psi \in C^1(X,\R)$, we see that a critical point $u \neq 0$ of $\Psi$ belongs to the set
\begin{align*}
	\mathcal{N} =\Big\{u \in X\setminus\{0\} \,:\, \lan \Psi'(u),u\ran=0 \Big\},
\end{align*}
where $\lan\cdot,\cdot\ran$ is the duality paring between $X$ and its dual space $X^*$. Therefore, $\mathcal{N}$ is a natural constraint for finding nontrivial critical points of $\Psi$. We mention the book chapter of Szulkin-Weth \cite{Szulkin-Weth-2010} in order to have a very well description of the method.

Because of the appearance of the singular term in \eqref{problem}, it is clear that the corresponding energy functional for problem \eqref{problem} is not $C^1$ and so we need to make several modifications in order to use the Nehari manifold method. With our work we extend the recent papers of Papageorgiou-Repov\v{s}-Vetro \cite{Papageorgiou-Repovs-Vetro-2021} for the weighted $(p,q)$-Laplacian and Papageorgiou-Winkert \cite{Papageorgiou-Winkert-2021} for the $p$-Laplacian. In contrast to these works we are working in Musielak-Orlicz Sobolev spaces and not in usual Sobolev spaces.

To the best of our knowledge, there are only two works dealing with singular double phase problems. Chen-Ge-Wen-Cao \cite{Chen-Ge-Wen-Cao-2020} considered problems of type \eqref{problem} and proved the existence of a weak solution with negative energy. Very recently, Farkas-Winkert \cite{Farkas-Winkert-2020} studied singular Finsler double phase problems of the form
\begin{equation*}
		-\divergenz\left(F^{p-1}(\nabla u)\nabla F(\nabla u) +\mu(x) F^{q-1}(\nabla u)\nabla F(\nabla u)\right) = u^{p^{*}-1}+\lambda \l(u^{\gamma-1}+g(u)\r)
\end{equation*}
in $\Omega$ and $u=0$ on $\partial\Omega$, where $(\R^N,F)$ is a Minkowski space. Based on variational tools, the existence of one weak solution is shown. Both works show only the existence of one weak solution (in contrast to our work) and the treatments are completely different from ours.

Finally, existence results for double phase problems with homogeneous Dirichlet or nonlinear Neumann boundary conditions without singular term can be found in the papers of Colasuonno-Squassina \cite{Colasuonno-Squassina-2016}, El Manouni-Marino-Winkert \cite{El-Manouni-Marino-Winkert-2020}, Gasi\'nski-Papa\-georgiou \cite{Gasinski-Papageorgiou-2019}, Gasi\'nski-Winkert \cite{Gasinski-Winkert-2020a,Gasinski-Winkert-2020b,Gasinski-Winkert-2021}, Liu-Dai \cite{Liu-Dai-2018,Liu-Dai-2020,Liu-Dai-2018b}, Marino-Winkert \cite{Marino-Winkert-2020}, Papageorgiou-R\u{a}dulescu-Repov\v{s} \cite{Papageorgiou-Radulescu-Repovs-2020d}, Papageorgiou-Vetro-Vetro \cite{Papageorgiou-Vetro-Vetro-2020}, Perera-Squassina \cite{Perera-Squassina-2019}, Zeng-Bai-Gasi\'nski-Winkert \cite{Zeng-Bai-Gasinski-Winkert-2020, Zeng-Gasinski-Winkert-Bai-2020} and the references there\-in. For related works dealing with certain types of double phase problems we refer to the works of Alves-Santos-Silva \cite{Alves-Santos-Silva-2021}, Bahrouni-R\u{a}dulescu-Winkert \cite{Bahrouni-Radulescu-Winkert-2020}, Barletta-Tornatore \cite{Barletta-Tornatore-2021}, Biagi-Esposito-Vecchi \cite{Biagi-Esposito-Vecchi-2021}, Lei \cite{Lei-2018}, Papageorgiou-R\u{a}dulescu-Repov\v{s} \cite{Papageorgiou-Radulescu-Repovs-2019b}, R\u{a}dulescu \cite{Radulescu-2019}, Sun-Wu-Long \cite{Sun-Wu-Long-2001}, Wang-Zhao-Zhao \cite{Wang-Zhao-Zhao-2013} and Zeng-Bai-Gasi\'nski-Winkert \cite{Zeng-Bai-Gasinski-Winkert-2020b}.

\section{Preliminaries }\label{section_2}

In this section we recall the main properties on the theory of Musielak-Orlicz spaces $\Lp{\mathcal{H}}$ and $\WH$, respectively. We refer to Colasuonno-Squassina \cite{Colasuonno-Squassina-2016}, Harjulehto-H\"{a}st\"{o} \cite{Harjulehto-Hasto-2019} and Musielak \cite{Musielak-1983} for the main results in this direction.

We denote by $\Lp{r}$ and $L^r(\Omega;\R^N)$ the usual Lebesgue spaces equipped with the norm $\|\cdot\|_r$ for every $1\leq r<\infty$.  For $1<r<\infty$, $\Wp{r}$ and $\Wpzero{r}$ stand for the Sobolev spaces endowed with the norms $\|\cdot \|_{1,r}$ and $\|\cdot\|_{1,r,0}=\|\nabla \cdot\|_r$, respectively.

Let $\mathcal{H}\colon \Omega \times [0,\infty)\to [0,\infty)$ be the function defined by
\begin{align*}
	\mathcal H(x,t)= t^p+\mu(x)t^q.
\end{align*}
Then, the Musielak-Orlicz space $L^\mathcal{H}(\Omega)$ is defined by
\begin{align*}
	L^\mathcal{H}(\Omega)=\left \{u ~ \Big | ~ u\colon \Omega \to \R \text{ is measurable and } \rho_{\mathcal{H}}(u)<+\infty \right \}
\end{align*}
equipped with the Luxemburg norm
\begin{align*}
	\|u\|_{\mathcal{H}} = \inf \left \{ \tau >0 : \rho_{\mathcal{H}}\left(\frac{u}{\tau}\right) \leq 1  \right \},
\end{align*}
where the modular function $\rho_{\mathcal{H}}\colon\Lp{\mathcal{H}}\to \R$ is given by
\begin{align}\label{modular}
	\rho_{\mathcal{H}}(u):=\into \mathcal{H}(x,|u|)\,\mathrm{d} x=\into \big(|u|^{p}+\mu(x)|u|^q\big)\,\mathrm{d} x.
\end{align}

From Colasuonno-Squassina \cite[Proposition 2.14]{Colasuonno-Squassina-2016} we know that the space $L^\mathcal{H}(\Omega)$ is a reflexive Banach space. Moreover, we define the seminormed space
\begin{align*}
	L^q_\mu(\Omega)=\left \{u ~ \Big | ~ u\colon \Omega \to \R \text{ is measurable and } \into \mu(x) |u|^q \,\mathrm{d} x< +\infty \right \},
\end{align*}
which is endowed with the seminorm
\begin{align*}
	\|u\|_{q,\mu} = \left(\into \mu(x) |u|^q \,\mathrm{d} x \right)^{\frac{1}{q}}.
\end{align*}
In the same way we define $L^q_\mu(\Omega;\R^N)$. 

The Musielak-Orlicz Sobolev space $W^{1,\mathcal{H}}(\Omega)$ is defined by
\begin{align*}
	W^{1,\mathcal{H}}(\Omega)= \left \{u \in L^\mathcal{H}(\Omega) \,:\, |\nabla u| \in L^{\mathcal{H}}(\Omega) \right\}
\end{align*}
equipped with the norm
\begin{align*}
	\|u\|_{1,\mathcal{H}}= \|\nabla u \|_{\mathcal{H}}+\|u\|_{\mathcal{H}},
\end{align*}
where $\|\nabla u\|_\mathcal{H}=\|\,|\nabla u|\,\|_{\mathcal{H}}$. The completion of $C^\infty_0(\Omega)$ in $W^{1,\mathcal{H}}(\Omega)$ is denoted by $W^{1,\mathcal{H}}_0(\Omega)$ and from \textnormal{(H)(i)} we have an equivalent norm on $W^{1,\mathcal{H}}_0(\Omega)$ given by
\begin{align*}
	\|u\|_{1,\mathcal{H},0}=\|\nabla u\|_{\mathcal{H}},
\end{align*}
see Proposition 2.18(ii) of Crespo-Blanco-Gasi\'nski-Harjulehto-Winkert \cite{Crespo-Blanco-Gasinski-Harjulehto-Winkert}. We know that $W^{1,\mathcal{H}}(\Omega)$ and $\WH$ are reflexive Banach spaces.

We have the following embedding results for the spaces $\Lp{\mathcal{H}}$ and $\Wpzero{\mathcal{H}}$.

\begin{proposition}\label{proposition_embeddings}
	Let \textnormal{(H)(i)} be satisfied. Then the following embeddings hold:
	\begin{enumerate}
		\item[\textnormal{(i)}]
		$\Lp{\mathcal{H}} \hookrightarrow \Lp{r}$ and $\WH\hookrightarrow \Wpzero{r}$ are continuous for all $r\in [1,p]$;
		\item[\textnormal{(ii)}]
		$\WH \hookrightarrow \Lp{r}$ is continuous for all $r \in [1,p^*]$;
		\item[\textnormal{(iii)}]
		$\WH \hookrightarrow \Lp{r}$ is compact for all $r \in [1,p^*)$;
		\item[\textnormal{(iv)}]
		$\Lp{\mathcal{H}} \hookrightarrow L^q_\mu(\Omega)$ is continuous;
		\item[\textnormal{(v)}]
		$\Lp{q} \hookrightarrow \Lp{\mathcal{H}}$ is continuous.
	\end{enumerate}
\end{proposition}

The norm $\|\cdot\|_{\mathcal{H}}$ and the modular function $\rho_\mathcal{H}$ are related as follows, see Liu-Dai \cite[Proposition 2.1]{Liu-Dai-2018}.

\begin{proposition}\label{proposition_modular_properties}
	Let \textnormal{(H)(i)} be satisfied, let $y\in \Lp{\mathcal{H}}$ and let $\rho_{\mathcal{H}}$ be defined by \eqref{modular}. Then the following hold:
	\begin{enumerate}
		\item[\textnormal{(i)}]
		If $y\neq 0$, then $\|y\|_{\mathcal{H}}=\lambda$ if and only if $ \rho_{\mathcal{H}}(\frac{y}{\lambda})=1$;
		\item[\textnormal{(ii)}]
		$\|y\|_{\mathcal{H}}<1$ (resp.\,$>1$, $=1$) if and only if $ \rho_{\mathcal{H}}(y)<1$ (resp.\,$>1$, $=1$);
		\item[\textnormal{(iii)}]
		If $\|y\|_{\mathcal{H}}<1$, then $\|y\|_{\mathcal{H}}^q\leq \rho_{\mathcal{H}}(y)\leq\|y\|_{\mathcal{H}}^p$;
		\item[\textnormal{(iv)}]
		If $\|y\|_{\mathcal{H}}>1$, then $\|y\|_{\mathcal{H}}^p\leq \rho_{\mathcal{H}}(y)\leq\|y\|_{\mathcal{H}}^q$;
		\item[\textnormal{(v)}]
		$\|y\|_{\mathcal{H}}\to 0$ if and only if $ \rho_{\mathcal{H}}(y)\to 0$;
		\item[\textnormal{(vi)}]
		$\|y\|_{\mathcal{H}}\to +\infty$ if and only if $ \rho_{\mathcal{H}}(y)\to +\infty$.
	\end{enumerate}
\end{proposition}

Let $A\colon \WH\to \WH^*$ be the nonlinear map defined by
\begin{align}\label{operator_representation}
	\langle A(u),\ph\rangle_{\mathcal{H}} :=\into \big(|\nabla u|^{p-2}\nabla u+\mu(x)|\nabla u|^{q-2}\nabla u \big)\cdot\nabla\ph \,\mathrm{d} x
\end{align}
for all $u,\ph\in\WH$, where $\lan\,\cdot\,,\,\cdot\,\ran_{\mathcal{H}}$ is the duality pairing between $\WH$ and its dual space $\WH^*$.  The operator $A\colon \WH\to \WH^*$ has the following properties, see Liu-Dai \cite{Liu-Dai-2018}.

\begin{proposition}
	The operator $A$ defined by \eqref{operator_representation} is bounded (that is, it maps bounded sets into bounded sets), continuous, strictly monotone (hence maximal monotone) and it is of type $(\Ss_+)$.
\end{proposition}

\section{Proof of the main result}

In this section we are going to prove our main result stated as Theorem \ref{main_result} in Section \ref{section_1}.

To this end, recall that $\ph_\lambda\colon\WH\to \R$ is the corresponding energy function for problem \eqref{problem} given by
\begin{align*}
		\ph_{\lambda}(u)=\frac{1}{p} \|\nabla u\|_p^p+\frac{1}{q}\|\nabla u\|_{q,\mu}^q-\frac{1}{1-\gamma}\into a(x)|u|^{1-\gamma}\,\mathrm{d} x-\frac{\lambda}{r}\|u\|_r^r.
\end{align*}
Due to the presence of the singular term $a(x)|u|^{1-\gamma}$ we know that $\ph_\lambda$ is not $C^1$. In order to overcome this, we will make use of the fibering method along with the Nehari manifold mentioned in the Introduction. Now we consider the fibering function $\omega_u\colon[0,+\infty)\to \R$ for $u\in \WH$ defined by 
\begin{align*}
	\omega_u(t)=\ph_\lambda (tu)\quad\text{for all }t\geq 0.
\end{align*}
The Nehari manifold corresponding to the functional $\ph_\lambda$ is defined by
\begin{align*}
	\mathcal{N}_\lambda
	&=\l\{u\in\WH\setminus\{0\}\,:\,\|\nabla u\|_p^p+\|\nabla u\|_{q,\mu}^q=\into a(x)|u|^{1-\gamma}\,\mathrm{d} x+\lambda \|u\|_r^r\r\}\\
	& =\l\{u\in\WH\setminus\{0\}\,:\, \omega_u'(1)=0\r\}.
\end{align*}
It is easy to see that $\mathcal{N}_\lambda$ is smaller than $\WH$ and it contains the weak solutions of problem \eqref{problem}. We will see that the functional $\ph_\lambda$ has nice properties restricted to $\mathcal{N}_\lambda$ which fail globally. For our further considerations we need to decompose the set $\mathcal{N}_\lambda$ in the following way:
\begin{align*}
	\mathcal{N}_\lambda^+
	&=\l\{u \in \mathcal{N}_\lambda: (p+\gamma-1)\|\nabla u\|_p^p+(q+\gamma-1)\|\nabla u\|_{q,\mu}^q-\lambda (r+\gamma-1) \|u\|_r^r>0\r\}\\
	& =\l\{u\in\mathcal{N}_\lambda\,:\, \omega_u''(1)>0\r\},\\
	\mathcal{N}_\lambda^0
	&=\l\{u \in \mathcal{N}_\lambda: (p+\gamma-1)\|\nabla u\|_p^p+(q+\gamma-1)\|\nabla u\|_{q,\mu}^q=\lambda (r+\gamma-1) \|u\|_r^r\r\}\\
	& =\l\{u\in\mathcal{N}_\lambda\,:\, \omega_u''(1)=0\r\},\\
	\mathcal{N}_\lambda^-
	&=\l\{u \in \mathcal{N}_\lambda: (p+\gamma-1)\|\nabla u\|_p^p+(q+\gamma-1)\|\nabla u\|_{q,\mu}^q-\lambda (r+\gamma-1) \|u\|_r^r<0\r\}\\
	& =\l\{u\in\mathcal{N}_\lambda\,:\, \omega_u''(1)<0\r\}.
\end{align*}

We start with the following proposition about the coercivity of the energy functional $\ph_\lambda$ restricted to $\mathcal{N}_\lambda$.

\begin{proposition}\label{prop_coerivity}
	Let hypotheses \textnormal{(H)} be satisfied. Then $\ph_\lambda\big|_{\mathcal{N}_\lambda}$ is coercive.
\end{proposition}

\begin{proof}
	Let $u \in \mathcal{N}_\lambda$ with $\|u\|_{1,\mathcal{H},0}>1$. From the definition of the Nehari manifold $\mathcal{N}_\lambda$ we have
	\begin{align}\label{prop_1}
		-\frac{\lambda}{r}\|u\|_r^r=-\frac{1}{r}  \|\nabla u\|_p^p-\frac{1}{r}\|\nabla u\|_{q,\mu}^q+\frac{1}{r}\into a(x)|u|^{1-\gamma}\,\mathrm{d} x.
	\end{align}
	Combining \eqref{prop_1} with $\ph_\lambda$ and applying Proposition \ref{proposition_modular_properties}\textnormal{(iv)} along with Theorem 13.17 of Hewitt-Stromberg \cite[p.\,196]{Hewitt-Stromberg-1965} gives
	\begin{align*}
		\ph_\lambda(u)&=\l[\frac{1}{p}-\frac{1}{r} \r]\|\nabla u\|_p^p+\l[\frac{1}{q}-\frac{1}{r} \r]\|\nabla u\|_{q,\mu}^q+\l[\frac{1}{r}-\frac{1}{1-\gamma} \r]\into a(x)|u|^{1-\gamma}\,\mathrm{d} x\\
		& \geq \l[\frac{1}{q}-\frac{1}{r} \r]\rho_\mathcal{H}(\nabla u) +\l[\frac{1}{r}-\frac{1}{1-\gamma} \r]\into a(x)|u|^{1-\gamma}\,\mathrm{d} x\\
		& \geq c_1 \|u\|_{1,\mathcal{H},0}^p-c_2\|u\|_{1,\mathcal{H},0}^{1-\gamma}
	\end{align*}
	for some $c_1,c_2>0$ because of $p<q<r$. Hence, due to $1-\gamma<1<p$, the assertion of the proposition follows. 
\end{proof}

Let $m_\lambda^+=\inf_{\mathcal{N}_\lambda^+}\ph_\lambda$.

\begin{proposition}\label{prop_negative_energy}
	Let hypotheses \textnormal{(H)} be satisfied and suppose that $\mathcal{N}_\lambda^+\neq \emptyset$. Then $m_\lambda^+<0$.
\end{proposition}

\begin{proof}
	Let $u \in \mathcal{N}_\lambda^+$. First note that $\mathcal{N}_\lambda^+\subseteq \mathcal{N}_\lambda$ which implies that
	\begin{align}\label{prop_3}
		-\frac{1}{1-\gamma}\into a(x)|u|^{1-\gamma}\,\mathrm{d} x =-\frac{1}{1-\gamma} \l(\|\nabla u\|_p^p+\|\nabla u\|_{q,\mu}^q\r)+\frac{\lambda}{1-\gamma}\|u\|_r^r.
	\end{align}
	On the other hand, by definition of $\mathcal{N}_\lambda^+$, we have
	\begin{align}\label{prop_2}
		\lambda \|u\|_r^r<\frac{p+\gamma-1}{r+\gamma-1}\|\nabla u\|_p^p+\frac{q+\gamma-1}{r+\gamma-1}\|\nabla u\|_{q,\mu}^q.
	\end{align}
	From \eqref{prop_2} and \eqref{prop_3} it follows that
	\begin{align*}
			\ph_{\lambda}(u)&=\frac{1}{p} \|\nabla u\|_p^p+\frac{1}{q}\|\nabla u\|_{q,\mu}^q-\frac{1}{1-\gamma}\into a(x)|u|^{1-\gamma}\,\mathrm{d} x-\frac{\lambda}{r}\|u\|_r^r\\
			& =\l[\frac{1}{p} -\frac{1}{1-\gamma}\r]\|\nabla u\|_p^p+\l[\frac{1}{q} -\frac{1}{1-\gamma}\r]\|\nabla u\|_{q,\mu}^q+\lambda \l[\frac{1}{1-\gamma}-\frac{1}{r}\r]\|u\|_r^r\\
			& \leq \l[\frac{-(p+\gamma-1)}{p(1-\gamma)}+\frac{p+\gamma-1}{r+\gamma-1}\cdot \frac{r+\gamma-1}{r(1-\gamma)}\r]\|\nabla u\|_p^p\\
			&\quad +\l[\frac{-(q+\gamma-1)}{q(1-\gamma)}+\frac{q+\gamma-1}{r+\gamma-1}\cdot \frac{r+\gamma-1}{r(1-\gamma)}\r]\|\nabla u\|_{q,\mu}^q\\
			& = \frac{p+\gamma-1}{1-\gamma}\l[\frac{1}{r}-\frac{1}{p}\r]\|\nabla u\|_p^p+\frac{q+\gamma-1}{1-\gamma}\l[\frac{1}{r}-\frac{1}{q}\r]\|\nabla u\|_{q,\mu}^q\\
			&<0, 
	\end{align*}
	since $p<q<r$. Hence, $\ph_\lambda \big|_{\mathcal{N}_\lambda^+}<0$ and so $m_\lambda^+<0$.
\end{proof}

\begin{proposition}\label{prop_emptiness}
	Let hypotheses \textnormal{(H)} be satisfied. Then there exists $\lambda^*>0$ such that $\mathcal{N}^0_\lambda=\emptyset$ for all $\lambda \in (0,\lambda^*)$.
\end{proposition}

\begin{proof}
	Arguing indirectly, suppose that  for every $\lambda^*>0$ there exists $\lambda \in (0,\lambda^*)$ such that $\mathcal{N}^0_\lambda \neq \emptyset$. Hence, for any given $\lambda>0$, we can find $u\in \mathcal{N}_\lambda^0$ such that
	\begin{align}\label{prop_4}
		(p+\gamma-1)\|\nabla u\|_p^p+(q+\gamma-1)\|\nabla u\|_{q,\mu}^q=\lambda (r+\gamma-1) \|u\|_r^r.
	\end{align}
	Since $u \in \mathcal{N}_\lambda$, one also has
	\begin{align}\label{prop_5}
		\begin{split}
		&(r+\gamma-1) \|\nabla u\|_p^p+(r+\gamma-1) \|\nabla u\|_{q,\mu}^q\\
		& = (r+\gamma-1)\into a(x)|u|^{1-\gamma}\,\mathrm{d} x+\lambda (r+\gamma-1)\|u\|_r^r.
		\end{split}
	\end{align}
	Subtracting \eqref{prop_4} from \eqref{prop_5} yields
	\begin{align}\label{prop_6}
			&(r-p) \|\nabla u\|_p^p+(r-q) \|\nabla u\|_{q,\mu}^q= (r+\gamma-1)\into a(x)|u|^{1-\gamma}\,\mathrm{d} x.
	\end{align}
	Applying Proposition \ref{proposition_modular_properties}\textnormal{(iii), (iv)}, Theorem 13.17 of Hewitt-Stromberg \cite[p.\,196]{Hewitt-Stromberg-1965} and Proposition \ref{proposition_embeddings}\textnormal{(ii)} we get from \eqref{prop_6} that
	\begin{align*}
		\min\l\{\|u\|_{1,\mathcal{H},0}^p,\|u\|_{1,\mathcal{H},0}^q\r\} \leq c_3 \|u\|_{1,\mathcal{H},0}^{1-\gamma}
	\end{align*}
	for some $c_3>0$ since $1-\gamma<1<p<q<r$. Hence
	\begin{align}\label{prop_7}
		\|u\|_{1,\mathcal{H},0} \leq c_4
	\end{align}
	for some $c_4>0$.

	On the other hand, from \eqref{prop_4}, Proposition \ref{proposition_modular_properties}\textnormal{(iii), (iv)} and Proposition \ref{proposition_embeddings}\textnormal{(ii)} we have
	\begin{align*}
		\min\l\{\|u\|_{1,\mathcal{H},0}^p,\|u\|_{1,\mathcal{H},0}^q\r\} \leq \lambda c_5 \|u\|_{1,\mathcal{H},0}^{r}
	\end{align*}
	for some $c_5>0$. Consequently,
	\begin{align*}
		\|u\|_{1,\mathcal{H},0} \geq \l(\frac{1}{\lambda c_5}\r)^{\frac{1}{r-p}}
		\quad\text{or}\quad
		\|u\|_{1,\mathcal{H},0} \geq \l(\frac{1}{\lambda c_5}\r)^{\frac{1}{r-q}}.
	\end{align*}
	If $\lambda \to 0^+$, due to $p<q<r$, then $\|u\|_{1,\mathcal{H},0}\to +\infty$, which contradicts \eqref{prop_7}.
\end{proof}

\begin{proposition}\label{prop_nonemptiness_and_existence}
	Let hypotheses \textnormal{(H)} be satisfied. Then there exists $\hat{\lambda}^*\in (0,\lambda^*]$ such that $\mathcal{N}^{\pm}_\lambda\neq \emptyset$ for all $\lambda \in (0,\hat{\lambda}^*)$. In addition, for any $\lambda \in (0,\hat{\lambda}^*)$, there exists $u^*\in\mathcal{N}_\lambda^+$ such that $\ph_\lambda(u^*)=m_\lambda^+<0$ and $u^*(x) \geq 0$ for a.\,a.\,$x\in\Omega$.
\end{proposition}

\begin{proof}
	Let $u\in \WH\setminus \{0\}$ and consider the function $\hat{\psi}_u\colon (0,+\infty) \to \R$ defined by
	\begin{align*}
		\hat{\psi}_u(t)= t^{p-r}\|\nabla u\|_p^p-t^{-r-\gamma+1}\into a(x)|u|^{1-\gamma} \,\mathrm{d} x.
	\end{align*}
	Since $r-p<r+\gamma-1$ we can find $\hat{t}_0>0$ such that
	\begin{align*}
		\hat{\psi}_u\l(\hat{t}_0\r)=\max_{t>0} \hat{\psi}_u(t).
	\end{align*}
	Thus, $\hat{\psi}'_u(\hat{t}_0)=0$, that is,
	\begin{align*}
		(p-r)\hat{t}_0^{p-r-1}\|\nabla u\|_p^p+(r+\gamma-1)\hat{t}_0^{-r-\gamma}\into a(x)|u|^{1-\gamma} \,\mathrm{d} x=0.
	\end{align*}
	Hence
	\begin{align*}
		\hat{t}_0=\l[\frac{(r+\gamma-1)\into a(x)|u|^{1-\gamma} \,\mathrm{d} x}{(r-p)\|\nabla u\|_p^p}\r]^{\frac{1}{p+\gamma-1}}.
	\end{align*}
	Moreover, we have
	\begin{align}
			\hat{\psi}_u\l(\hat{t}_0\r)
			&=\frac{\Big[(r-p) \|\nabla u\|_p^p \Big]^{\frac{r-p}{p+\gamma-1}}}{\Big[(r+\gamma-1)\into a(x)|u|^{1-\gamma}\,\mathrm{d} x \Big]^{\frac{r-p}{p+\gamma-1}}}\|\nabla u\|_p^p\nonumber\\
			&\quad -\frac{\Big[(r-p) \|\nabla u\|_p^p \Big]^{\frac{r+\gamma-1}{p+\gamma-1}}}{\Big[(r+\gamma-1)\into a(x)|u|^{1-\gamma}\,\mathrm{d} x \Big]^{\frac{r+\gamma-1}{p+\gamma-1}}}\into a(x)|u|^{1-\gamma}\,\mathrm{d} x\nonumber\\
			&=\frac{(r-p)^{\frac{r-p}{p+\gamma-1}} \|\nabla u\|_p^{\frac{p(r+\gamma-1)}{p+\gamma-1}}}{(r+\gamma-1)^{\frac{r-p}{p+\gamma-1}}\Big[\into a(x)|u|^{1-\gamma}\,\mathrm{d} x \Big]^{\frac{r-p}{p+\gamma-1}}}\label{prop_40}\\
			&\quad -\frac{(r-p)^{\frac{r+\gamma-1}{p+\gamma-1}} \|\nabla u\|_p^{\frac{p(r+\gamma-1)}{p+\gamma-1}}}{(r+\gamma-1)^{\frac{r+\gamma-1}{p+\gamma-1}}\Big[\into a(x)|u|^{1-\gamma}\,\mathrm{d} x \Big]^{\frac{r-p}{p+\gamma-1}}}\nonumber\\
			&=\frac{p+\gamma-1}{r-p} \l[\frac{r-p}{r+\gamma-1}\r]^{\frac{r+\gamma-1}{p+\gamma-1}}\frac{\|\nabla u\|_p^{\frac{p(r+\gamma-1)}{p+\gamma-1}}}{\Big[\into a(x)|u|^{1-\gamma}\,\mathrm{d} x \Big]^{\frac{r-p}{p+\gamma-1}}}.\nonumber
	\end{align}
	Let $S$ be the best constant of the Sobolev embedding $\Wpzero{p}\to \Lp{p^*}$, that is,
	\begin{align}\label{prop_41}
		S\|u\|_{p^*}^p \leq \|\nabla u\|_p^p.
	\end{align}
	Moreover, we have
	\begin{align}\label{prop_42}
		\into a(x)|u|^{1-\gamma}\,\mathrm{d} x \leq c_6 \|u\|_{p^*}^{1-\gamma}
	\end{align}
	for some $c_6>0$. Combining \eqref{prop_40}, \eqref{prop_41} and \eqref{prop_42} gives
	\begin{align*}
		&\hat{\psi}_u\l(\hat{t}_0\r)-\lambda \|u\|_r^r\\
		&=\frac{p+\gamma-1}{r-p} \l[\frac{r-p}{r+\gamma-1}\r]^{\frac{r+\gamma-1}{p+\gamma-1}}\frac{\|\nabla u\|_p^{\frac{p(r+\gamma-1)}{p+\gamma-1}}}{\Big[\into a(x)|u|^{1-\gamma}\,\mathrm{d} x \Big]^{\frac{r-p}{p+\gamma-1}}}-\lambda \|u\|_r^r\\
		&\geq \frac{p+\gamma-1}{r-p} \l[\frac{r-p}{r+\gamma-1}\r]^{\frac{r+\gamma-1}{p+\gamma-1}}\frac{S^{\frac{r+\gamma-1}{p+\gamma-1}}\l(\| u\|_{p^*}^p\r)^{\frac{r+\gamma-1}{p+\gamma-1}}}{\l(c_6 \|u\|_{p^*}^{1-\gamma}\r)^{\frac{r-p}{p+\gamma-1}}}-\lambda c_7\|u\|_{p^*}^r\\
		&= \Big [c_8-\lambda c_7\Big] \|u\|_{p^*}^r
	\end{align*}
	for some $c_7, c_8>0$. Therefore, there exists $\hat{\lambda}^* \in (0,\lambda^*]$ independent of $u$ such that
	\begin{align}\label{prop_43}
		\hat{\psi}_u\l(\hat{t}_0\r)-\lambda \|u\|_r^r>0 \quad\text{for all }\lambda \in \l(0,\hat{\lambda}^*\r).
	\end{align}

	Now consider the function $\psi_u\colon (0,+\infty) \to \R$ defined by
	\begin{align*}
		\psi_u(t)= t^{p-r}\|\nabla u\|_p^p+t^{q-r}\|\nabla u\|_{q,\mu}^q-t^{-r-\gamma+1}\into a(x)|u|^{1-\gamma} \,\mathrm{d} x.
	\end{align*}
	Since $r-q<r-p<r+\gamma-1$ we can find $t_0>0$ such that
	\begin{align*}
		\psi_u(t_0)=\max_{t>0} \psi_u(t).
	\end{align*}
	Because of $\psi_u \geq \hat{\psi}_u$ and due to \eqref{prop_43} (note that there the choice of $\hat{\lambda}^*$ is independent of $u$) we can find  $\hat{\lambda}^* \in (0,\lambda^*]$ independent of $u$ such that
	\begin{align*}
		\psi_u\l(t_0\r)-\lambda \|u\|_r^r>0 \quad\text{for all }\lambda \in \l(0,\hat{\lambda}^*\r).
	\end{align*}
	Thus there exist $t_1<t_0<t_2$ such that
	\begin{equation}\label{prop_8}
		\psi_u(t_1)=\lambda \|u\|_r^r=\psi_u(t_2)
		\quad \text{and}\quad 
		\psi'_u(t_2)<0<\psi'_u(t_1),
	\end{equation}
	where
	\begin{align}\label{prop_7b}
		\begin{split}
			\psi'_u(t)
			&=(p-r)t^{p-r-1}\|\nabla u\|_p^p+(q-r)t^{q-r-1}\|\nabla u\|_{q,\mu}^q\\
			&\quad -(-r-\gamma+1)t^{-r-\gamma}\into a(x)|u|^{1-\gamma}\,\mathrm{d} x.
		\end{split}
	\end{align}
	Note that $t_1,t_2$ are the only numbers which fulfill the equality in \eqref{prop_8}.

	Recall that the fibering function $\omega_u\colon[0,+\infty)\to \R$ is given by 
	\begin{align*}
		\omega_u(t)=\ph_\lambda (tu)\quad\text{for all }t\geq 0.
	\end{align*}
	Clearly, $\omega_u \in C^\infty((0,\infty))$. We have
	\begin{align*}
		\omega'_u(t_1)&=t_1^{p-1}\|\nabla u\|_p^p+t_1^{q-1}\|\nabla u\|_{q,\mu}^q -t_1^{-\gamma}\into a(x)|u|^{1-\gamma}\,\mathrm{d} x-\lambda t_1^{r-1}\|u\|_r^r
	\end{align*}
	and
	\begin{align}\label{prop_9}
		\begin{split}
		\omega''_u(t_1)&=(p-1)t_1^{p-2}\|\nabla u\|_p^p+(q-1)t_1^{q-2}\|\nabla u\|_{q,\mu}^q\\ 
		&\quad +\gamma t_1^{-\gamma-1}\into a(x)|u|^{1-\gamma}\,\mathrm{d} x
		-\lambda(r-1) t_1^{r-2}\|u\|_r^r.
		\end{split}
	\end{align}

	From \eqref{prop_8} we obtain
	\begin{align*}
		t_1^{p-r}\|\nabla u\|_p^p +t_1^{q-r}\|\nabla u\|_{q,\mu}^q-t_1^{-r-\gamma+1}\into a(x)|u|^{1-\gamma}\,\mathrm{d} x=\lambda \|u\|_r^r,
	\end{align*}
	which implies by multiplying with $\gamma t_1^{r-2}$ and $-(r-1)t_1^{r-2}$, respectively, that
	\begin{align}\label{prop_10}
		\gamma t_1^{p-2}\|\nabla u\|_p^p +\gamma t_1^{q-2}\|\nabla u\|_{q,\mu}^q-\gamma\lambda t_1^{r-2} \|u\|_r^r= \gamma t_1^{-\gamma-1}\into a(x)|u|^{1-\gamma}\,\mathrm{d} x
	\end{align}
	and
	\begin{align}\label{prop_11}
		\begin{split}
			&-(r-1)t_1^{p-2}\|\nabla u\|_p^p -(r-1)t_1^{q-2}\|\nabla u\|_{q,\mu}^q\\
			&\quad +(r-1)t_1^{-\gamma-1}\into a(x)|u|^{1-\gamma}\,\mathrm{d} x\\
			&=-\lambda(r-1) t_1^{r-2} \|u\|_r^r.
		\end{split}	
	\end{align}
	
	Applying \eqref{prop_10}  in \eqref{prop_9} gives
	\begin{align}\label{prop_12}
		\begin{split}
			\omega''_u(t_1)&=(p+\gamma -1)t_1^{p-2}\|\nabla u\|_p^p+(q+\gamma-1)t_1^{q-2}\|\nabla u\|_{q,\mu}^q\\ 
			&\quad -\lambda(r+\gamma -1) t_1^{r-2}\|u\|_r^r\\
			&= t_1^{-2}\Big[(p+\gamma -1)t_1^{p}\|\nabla u\|_p^p+(q+\gamma-1)t_1^{q}\|\nabla u\|_{q,\mu}^q\\ 
			&\qquad\quad -\lambda(r+\gamma -1) t_1^{r}\|u\|_r^r\Big].
		\end{split}
	\end{align}

	On the other hand, applying \eqref{prop_11}  in \eqref{prop_9} and using the representation in \eqref{prop_7b} leads to
	\begin{equation}\label{prop_13}
		\begin{split}
			\omega''_u(t_1)&=(p-r)t_1^{p-2}\|\nabla u\|_p^p+(q-r)t_1^{q-2}\|\nabla u\|_{q,\mu}^q\\ 
			&\quad +(r+\gamma-1) t_1^{-\gamma-1}\into a(x)|u|^{1-\gamma}\,\mathrm{d} x\\
			&=t^{1-r}_1 \psi'_u(t_1)>0.
		\end{split}
	\end{equation}
	From \eqref{prop_12} and \eqref{prop_13} it follows that
	\begin{align*}
		\begin{split}
			(p+\gamma -1)t_1^{p}\|\nabla u\|_p^p+(q+\gamma-1)t_1^{q}\|\nabla u\|_{q,\mu}^q-\lambda(r+\gamma -1) t_1^{r}\|u\|_r^r>0,
		\end{split}
	\end{align*}
	which implies
	\begin{equation*}
		t_1 u\in \mathcal{N}_\lambda^+ \quad \text{for all } \lambda\in \l(0,\hat{\lambda}^*\r].
	\end{equation*}
	Hence, $\mathcal{N}_\lambda^+\neq \emptyset$.
	
	Using similar arguments for the point $t_2$ (see \eqref{prop_8}), we can show that $\mathcal{N}_\lambda^-\neq \emptyset$. This shows the first assertion of the proposition. Let us now prove the second one.
	
	To this end, let $\{u_n\}_{n\in\N}\subset \mathcal{N}_\lambda^+$ be a minimizing sequence, that is,
	\begin{align}\label{prop_14}
		\ph_\lambda(u_n) \searrow m^+_\lambda <0 \quad\text{as }n\to\infty.
	\end{align}
	Recall that $\mathcal{N}_\lambda^+ \subset \mathcal{N}_\lambda$ and so we conclude from Proposition \ref{prop_coerivity} that $\{u_n\}_{n\in\N}\subset \WH$ is bounded. Therefore, we may assume that
	\begin{align}\label{prop_15}
		u_n\weak u^* \quad\text{in }\WH 
		\quad\text{and}\quad
		u_n\to u^* \quad\text{in }\Lp{r}.
	\end{align}
	From \eqref{prop_14} and \eqref{prop_15} we know that
	\begin{align*}
		\ph_\lambda(u^*) \leq \liminf_{n\to+\infty} \ph_\lambda(u_n)<0=\ph_\lambda(0).
	\end{align*}
	Hence, $u^*\neq 0$.
	
	{\bf Claim:} $\liminf_{n\to+\infty} \rho_{\mathcal{H}}(u_n)=\rho_{\mathcal{H}}(u^*)$
	
	Suppose, by contradiction, that 
	\begin{align*}
		\liminf_{n\to+\infty} \rho_{\mathcal{H}}(u_n)>\rho_{\mathcal{H}}(u^*).
	\end{align*}
	Then, by using \eqref{prop_8}, we have
	\begin{align*}
			&\liminf_{n\to+\infty} \omega'_{u_n}(t_1)\\
			&=\liminf_{n\to+\infty}\l[t_1^{p-1}\|\nabla u_n\|_p^p+t_1^{q-1}\|\nabla u_n\|_{q,\mu}^q -t_1^{-\gamma}\into a(x)|u_n|^{1-\gamma}\,\mathrm{d} x-\lambda t_1^{r-1}\|u_n\|_r^r\r]\\
			&>t_1^{p-1}\|\nabla u^*\|_p^p+t_1^{q-1}\|\nabla u^*\|_{q,\mu}^q -t_1^{-\gamma}\into a(x)|u^*|^{1-\gamma}\,\mathrm{d} x-\lambda t_1^{r-1}\|u^*\|_r^r\\
			&=\omega'_{u^*}(t_1)=t^{r-1}_1\l[\psi_{u^*}(t_1)-\lambda\|u^*\|_r^r\r]=0,
	\end{align*}
	which implies the existence of $n_0\in \N$ such that $\omega'_{u_n}(t_1)>0$ for all $n>n_0$. Recall that $u_n\in \mathcal{N}^+_{\lambda}\subset \mathcal{N}_{\lambda}$ and $\omega'_{u_n}(t)=t^{r-1} \l[\psi_{u_n}(t)-\lambda\|u_n\|_r^r\r]$. Thus we have $\omega'_{u_n}(t)<0$ for all $t\in(0,1)$ and $\omega'_{u_n}(1)=0$. Therefore, $t_1>1$.
	
	Since $\omega_{u^*}$ is decreasing on $(0,t_1]$, we have 
	\begin{align*}
		\ph_{\lambda} \l(t_1 u^*\r) \leq \ph_{\lambda}\l(u^*\r)<m^+_{\lambda}.
	\end{align*}
	Recall that $t_1 u^*\in \nc^+_{\lambda}$. So we obtain that 
	\begin{align*}
		m^+_{\lambda}\leq \ph_{\lambda}\l(t_1 u^*\r)<m^+_{\lambda},
	\end{align*}
	a contradiction. So the Claim is proved.
	
	From the Claim we know that we can find a subsequence (still denoted by $u_n$) such that $\rho_{\eh}\ykh{u_n}\to\rho_{\eh}\ykh{u^*}$. It follows from Proposition \ref{proposition_modular_properties}\textnormal{(v)} that $u_n \to u$ in $\WH$. This implies $\ph_{\lambda}(u_n)\to \ph_{\lambda}(u^*)$, and consequently, $\ph_{\lambda}(u^*)=m^+_{\lambda}$. Since $u_n\in \nc^+_{\lambda}$ for all $n\in \N$, we have
	\begin{align*}
		(p+\gamma-1)\|\nabla u_n\|_p^p+(q+\gamma-1)\|\nabla u_n\|_{q,\mu}^q-\lambda (r+\gamma-1) \|u_n\|_r^r>0.
	\end{align*}
	Letting $n\to+\infty$ gives
	\begin{equation}\label{prop_16}
		(p+\gamma-1)\|\nabla u^*\|_p^p+(q+\gamma-1)\|\nabla u^*\|_{q,\mu}^q-\lambda (r+\gamma-1) \|u^*\|_r^r\geq 0.
	\end{equation}
	Recall that $\lambda\in (0,\hat{\lambda}^*)$ and $\hat{\lambda}^*\leq \lambda^*$. Then, from Proposition \ref{prop_emptiness} we know that equality in \eqref{prop_16} cannot occur. Therefore, we conclude that  $u^*\in \mathcal{N}^+_{\lambda}$. Since we can use $|u^*|$ instead of $u^*$, we may assume that $u^*(x)\geq 0$ for a.\,a.\,$x\in\Omega$ with $u^*\neq 0$.  The proof is finished.
\end{proof}

In what follows, for $\eps>0$, we denote
\begin{align*}
	B_\eps(0)=\l\{u\in\WH\,:\, \|u\|_{1,\mathcal{H},0}<\eps\r\}.
\end{align*}
The next lemma is motivated by Lemma 3 of Sun-Wu-Long \cite{Sun-Wu-Long-2001}. This lemma is helpful in order to show that $u^*$ is a local minimizer of $\varphi_\lambda$ (see Proposition \ref{prop_energy_estimate}) and from this we conclude that $u^*$ is a weak solution of \eqref{problem} (see Proposition \ref{prop_first_weak_solution}).

\begin{lemma}\label{lemma_continuous_function}
	Let hypotheses \textnormal{(H)} be satisfied and let $u\in\mathcal{N}_\lambda^{\pm}$. Then there exist $\eps>0$ and a continuous function $\vartheta\colon B_{\e}(0)\to (0, \infty)$ such that 
	\begin{align*}
		\vartheta(0)=1
		\quad\text{and}\quad
		\vartheta(y)(u+y)\in \mathcal{N}_{\lambda}^{\pm}\quad \text{for all } y\in B_{\eps}(0).
	\end{align*}
\end{lemma}

\begin{proof}
	We show the proof only for $\mathcal{N}_\lambda^{+}$, the proof for $\mathcal{N}_\lambda^{-}$ works in a similar way. To this end, let $\zeta\colon\WH\times (0,\infty)\to\R$ be defined by
	\begin{align*}
		\zeta(y,t)
		&=t^{p+\gamma-1}\|\nabla (u+y)\|_p^p+t^{q+\gamma-1}\|\nabla (u+y)\|_{q,\mu}^q-\into a(x)|u+y|^{1-\gamma}\,\mathrm{d} x\\
		&\quad -\lambda t^{r+\gamma-1} \|u+y\|_r^r\quad \text{for all } y\in \WH.
	\end{align*}
	Since $u\in \mathcal{N}^+_\lambda \subset \mathcal{N}_\lambda$, one has $\zeta(0,1)=0$. Because of $u \in \mathcal{N}^+_\lambda$, it follows that
	\begin{align*}
		\zeta'_t(0,1)=(p+\gamma-1)\|\nabla u\|_p^p+(q+\gamma-1)\|\nabla u\|_{q,\mu}^q -\lambda (r+\gamma-1)\|u\|_r^r>0.
	\end{align*}
	Then, by the implicit function theorem, see, for example, Gasi\'{n}ski-Papageorgiou \cite[p.\,481]{Gasinski-Papageorgiou-2006}, there exist $\eps>0$ and a continuous function $\vartheta\colon B_\eps(0)\to (0,\infty)$ such that
	\begin{align*}
		\vartheta(0)=1\quad\text{and}\quad \vartheta(y)(u+y) \in \mathcal{N}_\lambda \quad\text{for all } y\in B_\eps(0).
	\end{align*}
	Choosing $\eps>0$ small enough, we also have 
	\begin{align*}
		\vartheta(0)=1\quad\text{and}\quad \vartheta(y)(u+y) \in N^{+}_\lambda \quad\text{for all }y\in B_\eps(0).
	\end{align*}
\end{proof}

\begin{proposition}\label{prop_energy_estimate}
	Let hypotheses \textnormal{(H)} be satisfied, let $h\in\WH$ and let $\lambda \in(0, \hat{\lambda}^*]$. Then there exists $b>0$  such that $\ph_{\lambda}(u^*)\leq \ph_{\lambda}(u^*+th)$ for all $t\in [0,b]$.
\end{proposition}

\begin{proof}
	We introduce the function $\eta_h\colon [0,+\infty)\to \R$ defined by
	\begin{align}\label{prop_17}
		\begin{split}
			\eta_h(t)
			&=(p-1)\l \|\nabla u^*+t\nabla h\r\|_p^p+(q-1)\|\nabla u^*+t\nabla h\|_{q,\mu}^q\\
			& \quad+\gamma \into a(x)\l|u^*+th \r|^{1-\gamma}\,\mathrm{d} x-\lambda (r-1)\l\| u^*+th \r\|_r^r.
		\end{split}
	\end{align}
	Recall that $u^* \in \mathcal{N}_\lambda^+\subseteq \mathcal{N}_\lambda$, see Proposition \ref{prop_nonemptiness_and_existence}. This implies
	\begin{align}\label{prop_18}
		\gamma \into a(x)\l| u^*\r|^{1-\gamma} \,\mathrm{d} x=\gamma \l\| \nabla u^*\r\|_p^p+\gamma \l\|\nabla u^*\r\|_{q,\mu}^q-\lambda \gamma \l\| u^*\r\|_r^r
	\end{align}
	and
	\begin{align}\label{prop_19}
		(p+\gamma-1)\l\| \nabla u^*\r\|_p^p+(q+\gamma-1) \l\|\nabla u^*\r\|_{q,\mu}^q-\lambda(r+\gamma-1) \l\| u^*\r\|_r^r>0.
	\end{align}
	Combining \eqref{prop_17}, \eqref{prop_18} and \eqref{prop_19} we see that $\eta_h(0)>0$. Since $\eta_h\colon [0,+\infty)\to \R$ is continuous we can find $b_0>0$ such that
	\begin{align*}
		\eta_h(t)>0 \quad\text{for all }t \in [0,b_0].
	\end{align*}
	Lemma \ref{lemma_continuous_function} implies that for every $t \in [0,b_0]$ we can find $\vartheta(t)>0$ such that
	\begin{align}\label{prop_20}
		\vartheta(t)\l(u^*+th\r)\in \mathcal{N}_\lambda^+
		\quad\text{and}\quad
		\vartheta(t) \to 1 \quad\text{as } t\to 0^+.
	\end{align}
	From Proposition \ref{prop_nonemptiness_and_existence} we know that
	\begin{align}\label{ineq-30}
		m_\lambda^+=\ph_\lambda \l(u^*\r) \leq \ph_\lambda \l(\vartheta(t)\l(u^*+th\r)\r)\quad\text{for all } t \in [0,b_0].
	\end{align}

	From $\omega_{u^*}''(1)>0$ and the continuity in $t$, we have $\omega''_{u^* +th}(1)>0$ for $t \in [0,b]$ with $b\in (0,b_0]$. Combining this with \eqref{ineq-30} gives
	\begin{align*}
		m_\lambda^+
		=\ph_\lambda \l(u^*\r) 
		\leq \ph_\lambda \l(\vartheta(t)\l(u^*+th\r)\r)
		=\omega_{u^*+th}(\vartheta(t))
		\leq \omega_{u^* +th}(1)
		=\ph_\lambda \l(u^*+th\r)
	\end{align*}
	for all $t \in [0,b]$.
\end{proof}

The next proposition shows that $\mathcal{N}^+_\lambda$ is a natural constraint for the energy functional $\ph_\lambda$, see  Papageorgiou-R\u{a}dulescu-Repov\v{s} \cite[p.\,425]{Papageorgiou-Radulescu-Repovs-2019}.

\begin{proposition}\label{prop_first_weak_solution}
	Let hypotheses \textnormal{(H)} be satisfied and let $\lambda \in(0, \hat{\lambda}^*]$. Then $u^*$ is a weak solution of problem \eqref{problem} such that $\ph_\lambda(u^*)<0$.
\end{proposition}

\begin{proof}
	From Proposition \ref{prop_nonemptiness_and_existence} we know that $u^*\geq 0$ for a.\,a.\,$x\in\Omega$ and $\ph_{\lambda}(u^*)<0$. 
	
	Let us prove that $u^*> 0$ for a.\,a.\,$x\in\Omega$. We argue indirectly and suppose there is a set $D$ with positive measure such that $u^*(x)=0$ for a.\,a.\,$x\in D$. Now let $h\in \WH$ with $h > 0$ and let $t\in (0,b)$, where $b$ is from Proposition \ref{prop_energy_estimate}. Then we have $(u^*+th)^{1-\gamma}>(u^*)^{1-\gamma}$ for a.\,a.\,$x\in D$. Applying this fact along with Proposition \ref{prop_energy_estimate} results in
	\begin{align*}
		0
		&\leq \frac{\ph_\lambda(u^*+th)-\ph_\lambda(u^*)}{t} \\
		&=\frac{1}{p} \frac{\|\nabla (u^*+th)\|_{p}^p-\|\nabla u^*\|_{p}^p}{t} +\frac{1}{q} \frac{\|\nabla (u^*+th)\|_{q,\mu}^q-\|\nabla u^*\|_{q,\mu}^q}{t}\\
		& \quad -\frac{1}{(1-\gamma)t^{\gamma}} \int_D a(x)h^{1-\gamma} \,\mathrm{d} x
		- \frac{1}{1-\gamma}\int_{\Omega\setminus D} a(x)\frac{(u^*+th)^{1-\gamma}-(u^*)^{1-\gamma}}{t}\,\mathrm{d} x\\
		&\quad -\frac{\lambda}{r}
		\frac{\|u^*+th\|_{r}^{r}-\|u^*\|_{r}^{r}}{t}\\
		&<\frac{1}{p} \frac{\|\nabla (u^*+th)\|_{p}^p-\|\nabla u^*\|_{p}^p}{t} +\frac{1}{q} \frac{\|\nabla (u^*+th)\|_{q,\mu}^q-\|\nabla u^*\|_{q,\mu}^q}{t}\\
		& \quad -\frac{1}{(1-\gamma)t^{\gamma}} \int_D a(x)h^{1-\gamma} \,\mathrm{d} x
		-\frac{\lambda}{r}
		\frac{\|u^*+th\|_{r}^{r}-\|u^*\|_{r}^{r}}{t}.
	\end{align*}
	Since $a>0$, see hypothesis \textnormal{(H)(iii)}, we conclude from the estimate above that
	\begin{align*}
		0
		&\leq \frac{\ph_\lambda(u^*+th)-\ph_\lambda(u^*)}{t} \to -\infty \quad \text{as }t\to 0^+.
	\end{align*}
	This is a contradiction and so we have that $u^*(x)>0$ for a.\,a.\,$x\in \Omega$.
	
	Next we prove that
	\begin{equation}\label{def1}
		a(\cdot)(u^*)^{-\gamma} h\in \Lp{1}  \quad \text{for all } h\in\WH
	\end{equation}
	and
	\begin{align}\label{def2}
		\begin{split}
			& \into \Big(|\nabla u^*|^{p-2} \nabla u^*+ \mu(x) |\nabla u^*|^{q-2} \nabla u^*\Big) \cdot \nabla h \,\mathrm{d} x\\
			&\geq \into a(x)(u^*)^{-\gamma} h \,\mathrm{d} x+\lambda \into (u^* )^{r-1}h\,\mathrm{d} x
		\end{split}
	\end{align}
	for all $h\in \WH$ with $h \geq 0$.
	
	To this end, let $h\in \WH$ with $h\geq 0$ and let $\{t_n\}_{n \in\N} \subseteq (0,1]$ be  a decreasing sequence such that $\displaystyle \lim_{n\to \infty} t_n=0$. First note that the functions
	\begin{align*}
		\kappa_n(x)=a(x)\frac{(u^*(x)+t_nh(x))^{1-\gamma}-u^*(x)^{1-\gamma}}{t_n}, \quad n\in\N
	\end{align*}
	are nonnegative and measurable. Furthermore, we have
	\begin{align*}
		\lim_{n\to \infty} \kappa_n(x)=(1-\gamma) a(x)u^*(x)^{-\gamma}h(x)\quad \text{for a.\,a.\,} x\in\Omega
	\end{align*}
	and by Fatou's lemma we get
	\begin{equation}\label{fatou}
		\into a(x) \l (u^*\r)^{-\gamma}h\,\mathrm{d} x\leq \frac{1}{1-\gamma}\liminf_{n\to\infty}\into \kappa_n\,\mathrm{d} x.
	\end{equation}
	Again from Proposition \ref{prop_energy_estimate} we get for $n\in\N$ sufficiently large that
	\begin{align*}
		0
		&\leq \frac{\ph_\lambda(u^*+t_nh)-\ph_\lambda(u^*)}{t_n} \\
		&=\frac{1}{p} \frac{\|\nabla (u^*+t_nh)\|_{p}^p-\|\nabla u^*\|_{p}^p}{t_n} +\frac{1}{q} \frac{\|\nabla (u^*+t_nh)\|_{q,\mu}^q-\|\nabla u^*\|_{q,\mu}^q}{t_n}\\
		&\quad - \frac{1}{1-\gamma}\into \kappa_n\,\mathrm{d} x-\frac{\lambda}{r}
		\frac{\|u^*+t_nh\|_{r}^{r}-\|u^*\|_{r}^{r}}{t_n}.
	\end{align*}
	If we pass to the limit as $n\to \infty$, taking \eqref{fatou} into account, we obtain
	\begin{align*}
			& \into a(x)(u^*)^{-\gamma} h \,\mathrm{d} x\\
			&\leq \into \Big(|\nabla u^*|^{p-2} \nabla u^*+ \mu(x) |\nabla u^*|^{q-2} \nabla u^*\Big) \cdot \nabla h \,\mathrm{d} x -\lambda \into (u^* )^{r-1}h\,\mathrm{d} x.
	\end{align*}
	This shows \eqref{def1}  and \eqref{def2}. We point out that it is sufficient to prove the integrability in \eqref{def1} for nonnegative test functions $h\in \WH$.
	
	In the next step we prove that $u^*$ is a weak solution of \eqref{problem}. Let $v\in\WH$ and let $\eps>0$. We take $h=(u^*+\eps v)_+$ as test function in \eqref{def2} and use  $u^*\in \mathcal{N}_\lambda^+\subset \mathcal{N}_\lambda$ with $u^*\geq 0$. One has
	\begin{align*}
		0& \le \into \ykh{\l|\nabla u^*\r|^{p-2}\nabla u^* + \mu(x)\l|\nabla u^*\r|^{q-2}\nabla u^*} \cdot \nabla (u^*+\eps v)_+ \,\mathrm{d} x\\
		&\quad -\into \l(a(x)\l(u^*\r)^{-\gamma}+\lambda \l(u^*\r)^{r-1}\r)(u^*+\eps v)_+\,\mathrm{d} x\\
		&=\int_{\{u^*+\eps v\geq 0\}} \ykh{\l|\nabla u^*\r|^{p-2}\nabla u^* + \mu(x)\l|\nabla u^*\r|^{q-2}\nabla u^*} \cdot \nabla (u^*+\eps v) \,\mathrm{d} x\\
		&\quad -\int_{\{u^*+\eps v\geq 0\}} \l(a(x)\l(u^*\r)^{-\gamma}+\lambda \l(u^*\r)^{r-1}\r)(u^*+\eps v)\,\mathrm{d} x\\
		&=\into\l(\l|\nabla u^*\r|^{p-2}\nabla u^* + \mu(x)\l|\nabla u^*\r|^{q-2}\nabla u^*\r) \cdot \nabla \l (u^*+\eps v\r) \,\mathrm{d} x\\
		&\quad - \int_{\l\{ u^*+\eps v< 0\r\}}\l(\l|\nabla u^*\r|^{p-2}\nabla u^* + \mu(x)\l|\nabla u^*\r|^{q-2}\nabla u^*\r) \cdot \nabla \l( u^*+\eps v \r) \,\mathrm{d} x\\
		& \quad -\into\l(a(x)\l(u^*\r)^{-\gamma}+\lambda \l(u^*\r)^{r-1}\r) \l( u^*+\eps v\r)\,\mathrm{d} x\\
		&\quad + \int_{\{ u^*+\e v< 0\}} \l(a(x)\l(u^*\r)^{-\gamma}+\lambda \l(u^*\r)^{r-1}\r)\l( u^*+\eps v \r)\,\mathrm{d} x\\
		&=\|\nabla u^*\|_{p}^p+\|\nabla u^*\|_{q,\mu}^q-\into a(x)|u^*|^{1-\gamma}\,\mathrm{d} x-\lambda \|u^*\|_{r}^{r}\\
		&\quad +\eps \into \l ( \l|\nabla u^*\r|^{p-2}\nabla u^* + \mu(x)\l|\nabla u^*\r|^{q-2}\nabla u^*\r) \cdot \nabla  v \,\mathrm{d} x\\
		&\quad -\eps \into  \l( a(x)\l(u^*\r)^{-\gamma}+\lambda \l(u^*\r)^{r-1}\r) v \,\mathrm{d} x\\
		&\quad - \int_{\l\{ u^*+\eps v< 0\r\}}\l(\l|\nabla u^*\r|^{p-2}\nabla u^* + \mu(x)\l|\nabla u^*\r|^{q-2}\nabla u^*\r) \cdot \nabla \l( u^*+\eps v \r) \,\mathrm{d} x\\
		&\quad +\int_{\{ u^*+\e v< 0\}} \l(a(x)\l(u^*\r)^{-\gamma}+\lambda \l(u^*\r)^{r-1}\r)\l( u^*+\eps v \r)\,\mathrm{d} x\\
		&\leq \eps \into \l ( \l|\nabla u^*\r|^{p-2}\nabla u^* + \mu(x)\l|\nabla u^*\r|^{q-2}\nabla u^*\r) \cdot \nabla  v \,\mathrm{d} x\\
		&\quad -\eps \into  \l( a(x)\l(u^*\r)^{-\gamma}+\lambda \l(u^*\r)^{r-1}\r) v \,\mathrm{d} x\\
		&\quad -\eps \int_{\l\{ u^*+\eps v< 0\r\}}\l(\l|\nabla u^*\r|^{p-2}\nabla u^* + \mu(x)\l|\nabla u^*\r|^{q-2}\nabla u^*\r) \cdot \nabla v \,\mathrm{d} x.
	\end{align*}
	Dividing the last inequality with $\eps>0$ and letting $\eps \to 0$, by taking 
	\begin{align*}
		& \int_{\l\{ u^*+\eps v< 0\r\}}\l(\l|\nabla u^*\r|^{p-2}\nabla u^* + \mu(x)\l|\nabla u^*\r|^{q-2}\nabla u^*\r) \cdot \nabla v \,\mathrm{d} x \to 0\quad\text{as }\eps \to 0
	\end{align*}
	into account, we obtain
	\begin{align*}
			& \into \Big(|\nabla u^*|^{p-2} \nabla u^*+ \mu(x) |\nabla u^*|^{q-2} \nabla u^*\Big) \cdot \nabla v \,\mathrm{d} x\\
			&\geq \into a(x)(u^*)^{-\gamma} v \,\mathrm{d} x+\lambda \into (u^* )^{\nu-1}v\,\mathrm{d} x.
	\end{align*}
	Since $v\in\WH$ is arbitrary chosen, equality must hold. It follows that $u^*$ is a weak solution of problem \eqref{problem} such that $\ph_\lambda(u^*)<0$, 
	see Propositions \ref{prop_negative_energy} and \ref{prop_nonemptiness_and_existence}.
\end{proof}

Now we start looking for a second weak solution when the parameter $\lambda>0$ is sufficiently small. To this end, we will use the manifold $\mathcal{N}^-_\lambda$.

\begin{proposition}\label{prop_minimizer_negative_manifold}
	Let hypotheses \textnormal{(H)} be satisfied. Then there exists $\hat{\lambda}_0^* \in(0, \hat{\lambda}^*]$ such that $\ph_\lambda\big|_{\mathcal{N}^-_\lambda} >0$ for all $\lambda \in(0, \hat{\lambda}^*_0]$.
\end{proposition}

\begin{proof}
	From Proposition \ref{prop_nonemptiness_and_existence} we know that $\mathcal{N}_\lambda^-\neq \emptyset$. Let $u \in \mathcal{N}_\lambda^-$. By the definition of $\mathcal{N}_\lambda^-$ and the embedding $\Wpzero{p}\to \Lp{r}$ we have
	\begin{align*}
			\lambda (r+\gamma-1) \|u\|_r^r&
			> (p+\gamma-1)\|\nabla u\|_p^p+(q+\gamma-1)\|\nabla u\|_{q,\mu}^q\\
			&\geq (p+\gamma-1) \|\nabla u\|_p^p\\
			& \geq (p+\gamma-1)c_9^p \|u\|_{r}^p
	\end{align*}
	for some $c_9>0$. Therefore
	\begin{equation}\label{prop_23}
		\|u\|_r\geq
		\displaystyle \l[\frac{c_9^p(p+\gamma-1)}{\lambda(r+\gamma-1)}\r]^{\frac{1}{r-p}}.
	\end{equation}

	Arguing by contradiction and suppose that the assertion of the proposition is not true. Then we can find $u \in \mathcal{N}_\lambda^-$ such that $\ph_\lambda(u)\leq 0$, that is,
	\begin{equation}\label{prop_24}
		\frac{1}{p}\|\nabla u\|_p^p +\frac{1}{q} \|\nabla u\|_{q,\mu}^q-\frac{1}{1-\gamma} \into a(x)|u|^{1-\gamma}\,\mathrm{d} x-\frac{\lambda}{r}\|u\|_r^r\leq 0.
	\end{equation}
	Since $\mathcal{N}_\lambda^-\subseteq \mathcal{N}_\lambda$ we know that
	\begin{equation}\label{prop_25}
		\frac{1}{q}\|\nabla u\|_{q,\mu}^q =\frac{1}{q}\into a(x)|u|^{1-\gamma}\,\mathrm{d} x+\frac{\lambda}{q} \|u\|_r^r -\frac{1}{q}\|\nabla u\|_p^p.
	\end{equation}
	Using \eqref{prop_25} in \eqref{prop_24} gives
	\begin{align*}
		\l(\frac{1}{p}-\frac{1}{q}\r)\|\nabla u\|_{p}^p+
		\l(\frac{1}{q}-\frac{1}{1-\gamma}\r)\into a(x)|u|^{1-\gamma}\,\mathrm{d} x+\lambda \l(\frac{1}{q}-\frac{1}{r}\r)\|u\|_r^r\leq 0.
	\end{align*}
	This yields
	\begin{align*}
		\lambda \frac{r-q}{qr}\|u\|_r^r\leq \frac{q+\gamma -1}{q(1-\gamma)}\into a(x) |u|^{1-\gamma}\,\mathrm{d} x \leq \frac{q+\gamma -1}{q(1-\gamma)}c_{10}\|u\|_r^{1-\gamma}
	\end{align*}
	for some $c_{10}>0$. Therefore,
	\begin{equation}\label{prop_26}
		\|u\|_r \leq c_{11} \l(\frac{1}{\lambda}\r)^{\frac{1}{r+\gamma-1}}
	\end{equation}
	for some $c_{11}>0$. Now we use \eqref{prop_26}  in \eqref{prop_23} in order to obtain
	\begin{align*}
		c_{12}\l(\frac{1}{\lambda}\r)^{\frac{1}{r-p}}
		\leq c_{11} \l(\frac{1}{\lambda}\r)^{\frac{1}{r+\gamma-1}}\quad\text{with}
		\quad \displaystyle c_{12}=\l[\frac{c_9^p(p+\gamma-1)}{r+\gamma-1}\r]^{\frac{1}{r-p}}>0.
	\end{align*}
	Consequently
	\begin{align*}
		0<\frac{c_{12}}{c_{11}} \leq \lambda^{\frac{1}{r-p}-\frac{1}{r+\gamma-1}}=\lambda^{\frac{p+\gamma-1}{(r-p)(r+\gamma-1)}}\to 0\quad\text{as}\quad \lambda\to 0^+,
	\end{align*}
	since $1<p<r$ and $\gamma \in (0,1)$, a contradiction. Thus, we can find $\hat{\lambda}_0^* \in(0, \hat{\lambda}^*]$ such that $\ph_\lambda\big|_{\mathcal{N}^-_\lambda} > 0$ for all $\lambda \in(0, \hat{\lambda}^*_0]$.
\end{proof}

Now we minimize $\ph_\lambda$ on the manifold $\mathcal{N}_\lambda^-$.

\begin{proposition}\label{prop_minimize_negative}
	Let hypotheses \textnormal{(H)} be satisfied and let $\lambda \in(0, \hat{\lambda}^*_0]$. Then there exists $v^*\in\mathcal{N}_\lambda^-$ with $v^*\geq 0$ such that
	\begin{align*}
		m_\lambda^-=\inf_{\mathcal{N}_\lambda^-}\ph_\lambda=\ph_\lambda\l(v^*\r)>0.
	\end{align*}
\end{proposition}

\begin{proof}
	Let $\{v_n\}_{n\in\N}\subset \mathcal{N}_\lambda^- \subset \mathcal{N}_\lambda$ be a minimizing sequence. Since $\mathcal{N}_\lambda^-\subset \mathcal{N}_\lambda$, we know that $\{v_n\}_{n\in\N}\subset \WH$ is bounded, see Proposition \ref{prop_coerivity}. We may assume that
	\begin{align*}
		v_n\weak v^* \quad\text{in }\WH
		\quad\text{and}\quad v_n\to v^*\quad\text{in }\Lp{r}.
	\end{align*}
	Note that $v^* \neq 0$  by \eqref{prop_23}. 
	Now we will use the point $t_2>0$ (see \eqref{prop_8}) for which we have
	\begin{align*}
		\psi_{v^*}\l(t_2\r)=\lambda \l\|v^*\r\|_{r}^{r}\quad\text{and}\quad \psi'_{v^*}\l(t_2\r)<0.
	\end{align*}
	In the proof of Proposition \ref{prop_nonemptiness_and_existence} we showed that $t_2 v^* \in \mathcal{N}_\lambda^-$.
	
	Next we want to show that $\rho_{\mathcal{H}}(v_n)\to \rho_{\mathcal{H}}(v^*)$ as $n\to \infty$ for a subsequence (still denoted by $v_n$). Let us suppose this is not the case, then we have as in the proof of Proposition \ref{prop_nonemptiness_and_existence} that
	\begin{align*}
		\ph_\lambda(t_2 v^*)< \lim_{n\to\infty} \ph_\lambda(t_2 v_n).
	\end{align*}
	We know that $\ph_\lambda(t_2 v_n) \leq \ph_\lambda(v_n)$ since it is the global maximum because of $\omega_{v_n}''(1)<0$. Using this and  $t_2 v^* \in \mathcal{N}_\lambda^-$, we get
	\begin{align*}
		m_\lambda^- \leq \ph_\lambda(t_2 v^*) < m^-_\lambda,
	\end{align*}
	which is a contradiction. Hence we have $\lim_{n\to+\infty} \rho_{\mathcal{H}}(v_n)=\rho_{\mathcal{H}}(v^*)$ for a subsequence and so Proposition \ref{proposition_modular_properties}\textnormal{(v)} implies $v_n \to v^*$ in $\WH$. The continuity of $\ph_\lambda$ then gives  $\ph_{\lambda}(v_n)\to \ph_{\lambda}(v^*)$, thus, $\ph_{\lambda}(v^*)=m^-_{\lambda}$.

	Since $v_n\in \nc^-_{\lambda}$ for all $n\in \N$, we have
	\begin{align*}
		(p+\gamma-1)\|\nabla v_n\|_p^p+(q+\gamma-1)\|\nabla v_n\|_{q,\mu}^q-\lambda (r+\gamma-1) \|v_n\|_r^r<0.
	\end{align*}
	If we pass to the limit as $n\to+\infty$ we obtain
	\begin{equation}\label{prop_166}
		(p+\gamma-1)\|\nabla v^*\|_p^p+(q+\gamma-1)\|\nabla v^*\|_{q,\mu}^q-\lambda (r+\gamma-1) \|v^*\|_r^r\leq 0.
	\end{equation}
	Recall that $\lambda\in (0,\hat{\lambda}^*)$ and $\hat{\lambda}^*\leq \lambda^*$. Applying Proposition \ref{prop_emptiness} we see that equality in \eqref{prop_166} cannot happen. Hence, $v^*\in \mathcal{N}^-_{\lambda}$. Since the treatment also works for $|v^*|$ instead of $v^*$, we may assume that $v^*(x)\geq 0$ for a.\,a.\,$x\in\Omega$ such that $v^*\neq 0$. Proposition \ref{prop_minimizer_negative_manifold} finally shows that $m_\lambda^->0$.
\end{proof}

Now we have a second weak solution of problem \eqref{problem}.

\begin{proposition}\label{prop_seond_weak_solution}
	Let hypotheses \textnormal{(H)} be satisfied and let $\lambda \in(0, \hat{\lambda}^*_0]$. Then $v^*$ is a weak solution of problem \eqref{problem} such that $\ph_\lambda(v^*)>0$.
\end{proposition}

\begin{proof}
	Following the proof of Proposition \ref{prop_energy_estimate} replacing $u^*$ by $v^*$ in the definition of $\eta_h$ by using Lemma \ref{lemma_continuous_function} we are able to show that for every $t \in [0,b_0]$ there exists $\vartheta(t)>0$ such that
	\begin{align*}
		\vartheta(t)\l(v^*+th\r)\in \mathcal{N}_\lambda^-
		\quad\text{and}\quad
		\vartheta(t) \to 1 \quad\text{as } t\to 0^+,
	\end{align*}
	see also \eqref{prop_20}. Taking Proposition \ref{prop_minimize_negative} into account we have that
	\begin{align}\label{prop_200}
		m_\lambda^-=\ph_\lambda \l(v^*\r) \leq \ph_\lambda \l(\vartheta(t)\l(v^*+th\r)\r)\quad\text{for all } t \in [0,b_0].
	\end{align}
	
	Let us now show that  $v^*(x)> 0$ for a.\,a.\,$x\in\Omega$. As for $u^*$, let us suppose there exists a set $E$ with positive measure such that $v^*=0$ in $E$. Taking $h\in \WH$ with $h > 0$ and $t\in (0,b_0)$, we know that $(\vartheta(t)(v^*+th))^{1-\gamma}>(\vartheta(t)v^*)^{1-\gamma}$ a.\,e.\,in $\Omega\setminus E$. Note that $\omega_{v^*}(1)$ is the global maximum which implies $\omega_{v^*}(1) \geq \omega_{v^*}(\vartheta(t))$. Using this and  \eqref{prop_200} it follows that
	\begin{align*}
		0
		&\leq \frac{\ph_\lambda(\vartheta(t)(v^*+th))-\ph_\lambda(v^*)}{t} \\
		&\leq \frac{\ph_\lambda(\vartheta(t)(v^*+th))-\ph_\lambda(\vartheta(t)v^*)}{t}\\
		&=\frac{1}{p} \frac{\|\nabla (\vartheta(t)(v^*+th))\|_{p}^p-\|\nabla (\vartheta(t)v^*)\|_{p}^p}{t}\\ 
		&\quad +\frac{1}{q} \frac{\|\nabla (\vartheta(t)(v^*+th))\|_{q,\mu}^q-\|\nabla (\vartheta(t)v^*)\|_{q,\mu}^q}{t}-\frac{\vartheta(t)^{1-\gamma}}{(1-\gamma)t^{\gamma}} \int_E a(x)h^{1-\gamma} \,\mathrm{d} x\\
		&\quad - \frac{1}{1-\gamma}\int_{\Omega\setminus E} a(x)\frac{(\vartheta(t)(v^*+th))^{1-\gamma}-(\vartheta(t)v^*)^{1-\gamma}}{t}\,\mathrm{d} x\\
		&\quad -\frac{\lambda}{r}
		\frac{\|\vartheta(t)(v^*+th)\|_{r}^{r}-\|\vartheta(t)v^*\|_{r}^{r}}{t}\\
		&<\frac{1}{p} \frac{\|\nabla (\vartheta(t)(v^*+th))\|_{p}^p-\|\nabla (\vartheta(t)v^*)\|_{p}^p}{t}\\ 
		&\quad +\frac{1}{q} \frac{\|\nabla (\vartheta(t)(v^*+th))\|_{q,\mu}^q-\|\nabla (\vartheta(t)v^*)\|_{q,\mu}^q}{t}-\frac{\vartheta(t)^{1-\gamma}}{(1-\gamma)t^{\gamma}} \int_E a(x)h^{1-\gamma} \,\mathrm{d} x\\
		&\quad -\frac{\lambda}{r}
		\frac{\|\vartheta(t)(v^*+th)\|_{r}^{r}-\|\vartheta(t)v^*\|_{r}^{r}}{t}.
	\end{align*}
	Therefore, similar to the proof of  Proposition \ref{prop_first_weak_solution}, we see from the inequality above that \begin{align*}
		0
		&\leq \frac{\ph_\lambda(\vartheta(t)(v^*+th))-\ph_\lambda(\vartheta(t)v^*)}{t} \to -\infty \quad \text{as }t\to 0^+,
	\end{align*}
	which is again a contradiction. We  conclude that $v^*(x)>0$ for a.\,a.\,$x\in \Omega$.
	
	The rest of the proof can be done similarly as the proof of Proposition \ref{prop_first_weak_solution}. Precisely, \eqref{def1} and \eqref{def2} can be proven  in the same way by applying again  \eqref{prop_200} and the inequality $\omega_{v^*}(1) \geq \omega_{v^*}(\vartheta(t))$ together with $v^*>0$. Finally, the last part of Proposition \ref{prop_first_weak_solution} is the same replacing $u^*$ by $v^*$. From Proposition \ref{prop_minimize_negative} we know that $\ph_\lambda(v^*)>0$. This finishes the proof. 
\end{proof}

The proof of Theorem \ref{main_result} follows now from Propositions \ref{prop_first_weak_solution} and \ref{prop_seond_weak_solution}.

\section*{Acknowledgment}
The authors wish to thank Professor R.\,L.\,Alves for pointing out mistakes in the proofs of Propositions \ref{prop_energy_estimate} and \ref{prop_first_weak_solution} in the first version of the manuscript.

W. Liu was supported by the NNSF of China (Grant No. 11961030). 
G. Dai is supported by the NNSF of China (Grant No. 11871129), 
the Fundamental Research Funds for the Central Universities (Grant No. DUT17LK05), 
and Xinghai Youqing funds from Dalian University of Technology.


\end{document}